\documentclass{amsart}
\usepackage{tikz}
\usetikzlibrary{fadings,arrows}

\usepackage{amsmath, amssymb, mathrsfs, verbatim, multirow}
\usepackage[all]{xy}
\usepackage{pifont}
\usepackage{float}

\newtheorem{Teo}{Theorem}[section]
\newtheorem{Prop}[Teo]{Proposition}
\newtheorem{Lema}[Teo]{Lemma}

\theoremstyle{definition}
\newtheorem{Def}[Teo]{Definition}

\newtheorem{Obs}[Teo]{Remark}

\newcommand{\N}{\mathbb{N}}

\newcommand{\lra}{\longrightarrow}

\newcommand{\VR}{\mathcal{O}}

\newcommand{\rk}{\mbox{\rm rk}}

\newcommand{\vdeg}{\mbox{\rm vdeg}}

\begin{document}
\title[Valuation ring extensions]{Valuation rings in simple algebraic extensions of valued fields}

\author{Josnei Novacoski}
\address{Departamento de Matem\'{a}tica,  Universidade Federal de S\~ao Carlos, Rod. Washington Luís, 235, 13565--905, S\~ao Carlos -SP, Brazil}
\email{josnei@ufscar.br}
\author{Mark Spivakovsky}
\address{CNRS UMR 5219, Institut de Mathématiques de Toulouse, 118, rte de Narbonne, 31062 Toulouse cedex 9, France and 
Instituto de Matem\'aticas (Cuernavaca) LaSol, UMI CNRS 2001, UNAM, Av. Universidad s/n. Col. Lomas de Chamilpa, 62210, Cuernavaca, Morelos, M\'exico}
\email{mark.spivakovsky@math.univ-toulouse.fr}

\thanks{During the realization of this project the first author was supported by a grant from Funda\c{c}\~ao de Amparo \`a Pesquisa do Estado de S\~ao Paulo (process numbers 2017/17835-9 and 2019/21181-0) and a grant from Conselho Nacional de Desenvolvimento Cient\'ifico e Tecnol\'ogico (process number 303215/2022-4).}

\begin{abstract}
Consider a simple algebraic valued field extension $(L/K,v)$ and denote by $\VR_L$ and $\VR_K$ the corresponding valuation rings. The main goal of this paper is to present, under certain assumptions, a description of $\VR_L$ in terms of generators and relations over $\VR_K$. The main tool used here are complete sequences of key polynomials. It is known that if the ramification index of $(L/K,v)$ is one, then every complete set gives rise to a set of generators of $\VR_L$ over $\VR_K$. We show that we can find a sequence of key polynomials for $(L/K,v)$ which satisfies good properties (called neat). Then we present explicit ``neat" relations that generate all the relations between the corresponding generators of $\VR_L$ over
$\VR_K$.
\end{abstract}

\maketitle
\section{Introduction}

Consider a simple algebraic valued field extension $(L/K,v)$ and denote by $\VR_L$ and $\VR_K$ the corresponding valuation rings. The main goal of this paper is to present, under certain assumptions (\textbf{Assumptions 1--3} below), a description of $\VR_L$ in terms of generators and relations over $\VR_K$. \emph{Key polynomials} are very useful tools for this purpose.

Fix a generator $\eta$ of $L$ over $K$ and take a \emph{complete sequence of key polynomials} $\textbf{Q}=\{Q_i\}_{i\in I}$ for the valuation $\nu$ on $K[x]$ induced by $v$ and $\eta$:
\[
\nu(f):=v(f(\eta))\mbox{ for every }f\in K[x].
\]
In particular, $I$ has a largest element $i_{\max}$ and $g:=Q_{i_{\max}}$ is the monic minimal polynomial of $\eta$ over $K$. We set $I^*:=I\setminus\{i_{\rm max}\}$.

We will assume throughout this paper that
\begin{description}
\item[Assumption 1] $e(L/K,v):=[vL:vK]=1$.
\end{description}
Under this assumption, every \emph{complete set} for $\nu$ naturally gives rise to a set of generators for the extension
$\VR_L/\VR_K$ (see Section \ref{completesets}). In order to describe $\VR_L$ as an $\VR_K$-algebra, we need to write down a full set of relations between these generators.  The idea in this paper is to present a complete sequence of key polynomials in a way that the corresponding relations are as simple as possible.

The applications we have in mind are the study of \emph{graded algebras} associated to valuations and the computation of the module of \emph{K\"ahler differentials} of the extension $\VR_L/\VR_K$ (the latter task will be accomplished in a forthcoming paper). It is well known that complete sets are closely related to  sets of generators of the corresponding graded algebras. We hope that the results of this paper will help to provide simpler presentations for graded algebras. This is particularly important in B. Teissier's program for solving the local uniformization problem in positive characteristic (see \cite{Tei} and \cite{Tei2}).

Another particular interest is to give an explicit description of the module of K\"ahler differentials for the extension
$\VR_L/\VR_K$. This problem has been extensively  studied recently (see \cite{C}, \cite{CK}, \cite{CKR} and \cite{NovSpiAnn}). In order to present this module, one needs to have a description of $\VR_L$ in terms of generators and relations over $\VR_K$. This is precisely the goal of this paper.

Because of \textbf{Assumption 1}, multiplying $\eta$ by a suitable element of $K$, we may assume that $v\eta=\nu x=0$; we will make this assumption from now on.

Let $q,f\in K[x]$ be two polynomials with $\deg(q)\geq 1$. There exist (uniquely determined) $f_0,\ldots,f_n\in K[x]$ with $\deg(f_i)<\deg(q)$ for every $i$, $0\leq i\leq n$, such that
\begin{equation}\label{qexpan}
f=f_0+f_1q+\ldots+f_nq^n.
\end{equation}
This expression is called the \textbf{$q$-expansion of $f$}. If $f_n=1$ in the above expansion, then we say that $f$ is \textbf{$q$-monic}. For a monic polynomial $q$ of $\deg(q)\geq 1$, the \textbf{truncation of $\nu$ at $q$} is defined by
\[
\nu_q(f)=\min\left\{\nu\left(f_iq^i\right)\right\},
\]
where $f=f_0+f_1q+\ldots+f_nq^n$ is the $q$-expansion of $f$. Denote
\[
S_q(f)=\{i\in\{0,\ldots,n\}\mid \nu_q(f)=\nu(f_iq^i)\}.
\]
We will say that $f$ is \textbf{$q$-strongly monic} if
\[
f_n=1\mbox{ and }n\in S_q(f).
\]

For each $i\in I^*$ we denote by $\nu_i$ the truncation of $\nu$ at $Q_i$. For each $n\in\N$ we denote by $\textbf{Q}_n$ the set of elements of $\textbf{Q}$ of degree $n$. If $\textbf{Q}_n$ does not have a largest element (with respect to the pre-ordering induced by $\nu$), then we say that $\textbf{Q}_n$ is an infinite \textbf{plateau} (of key polynomials). For a strictly positive natural number $n$ such that $\textbf{Q}_n\ne\emptyset$, put
\[
n_+=\min\left\{n'\in\N\ \left|\ n'>n,\textbf{Q}_{n'}\ne\emptyset\right.\right\}.
\]
\begin{Obs}
With the above notation, if for each $n$ satisfying $\textbf Q_n\neq \emptyset$ we take a cofinal family $\textbf Q'_n$ of $\textbf Q_n$, then $\textbf{Q}'=\bigcup\textbf{Q}'_{n}$ is also a complete sequence of key polynomials for $\nu$. Hence, if $\textbf{Q}_n$ has a largest element, then we will assume that $\textbf{Q}_n$ is a singleton.
\end{Obs}
\begin{description}
\item[Assumption 2] Assume that every infinite plateau $\textbf{Q}_n$ has order type $\omega$.
\end{description}
\begin{Obs}
If an infinite plateau $\textbf{Q}_n$ admits a countable cofinal subset, then it contains a cofinal subset of order type $\omega$. Replacing $\textbf Q_n$ by this cofinal subset, we can assume that $\textbf Q_n$ itself has order type $\omega$. This condition is automatically satisfied, for instance, if the rank of $v$ is countable. For simplicity of exposition, we will make the above assumption from now till the end of the paper. More generally, the results of this paper hold (with almost identical proofs) if we assumed that all infinite plateau $\textbf{Q}_{n_i}$ contain cofinal subsets with the same cardinality.
\end{Obs}

Since for every key polynomial $Q$ we have $\deg_xQ\le\deg_xg$, the set of non-empty plateaus is finite. Let us decompose $\bf Q$ as a disjoint union
\[
{\bf Q}=\coprod\limits_{q=1}^{w+1}{\bf Q}_{n_q},
\]
where $1=n_1<n_2<\dots<n_{w+1}:=\deg_xg$. This induces a decomposition of the set $I$ as a disjoint union $I=\coprod\limits_{q=1}^{{w+1}}I^{(q)}$, where the $I^{(q)}$ are segments in $I^*$ with
\[
I^{(1)}<I^{(2)}<\dots<I^{(w)}<I^{(w+1)}=\{i_{\text{max}}\}
\]
and each $I^{(q)}$ is either a singleton or is order-isomorphic to the set $\N$ of natural numbers. Namely, we put $I^{(q)}=\left\{i\in I\ \left|\ \deg_xQ_i=n_q\right.\right\}$, $q\in\{1,\dots,w+1\}$. For each $q$ we denote by $\ell_q$ the smallest element of $I^{(q)}$.

For each $i\in I^*$, choose $a_i\in K$ such that $\nu(Q_i)=v(a_i)$ and set $\tilde{Q}_i:=Q_i/a_i$ (so that
$\nu\left(\tilde{Q}_i\right)=0$; in particular, we have $\tilde Q_{\ell_1}=x$). Put $\tilde
Q_{i_{\text{max}}}=Q_{i_{\text{max}}}=g$. Set $\tilde{\textbf{Q}}=\{\tilde Q_i\}_{i\in I}$. We will denote by $\N^I$ the set of all maps $\lambda: I\lra \N$ such that $\lambda(i)\neq 0$ for only finitely many $i\in I$. For each $\lambda\in \N^I$ we set
\[
\tilde{\textbf Q}^\lambda:=\prod_{i\in I} \tilde Q_i^{\lambda(i)}\in K[x]\mbox{ and }\textbf Q^\lambda:=\prod_{i\in I} Q_i^{\lambda(i)}\in K[x].
\]

Consider a polynomial $f\in K[x]\setminus\{0\}$. For $i\in I^*$ we will consider expressions of the form
\begin{equation}\label{expiadicf}
f=\sum_{j=1}^rb_j{\bf\tilde Q}^{\lambda_j}\mbox{ with }b_j\in K, \lambda_j\in \N^I\mbox{ and }\lambda_j\left(k\right)=0\mbox{ if }k>i.
\end{equation}
For an index $k\in I^*$, we say that $\tilde Q_k$ {\bf appears in \eqref{expiadicf}} if $\lambda_j\left(k\right)>0$ for some $j\in\{1,\dots,r\}$. A \textbf{full $i$-th expansion of $f$} is an expression \eqref{expiadicf} having the following properties:
\begin{enumerate}
\item $\displaystyle\nu_i(f)=\min_{0\leq j\leq r}\{v(b_j)\}$
\item if $k<i$ and $j\in\{1,\dots,r\}$, then $\lambda_j\left(k\right)<\left(\deg_xQ_k\right)_+$
\item for each $n\in\N_{>0}$ we have
\[
\#\{k\in I^*\ |\ \deg_xQ_k=n\mbox{ and }\tilde Q_k\mbox{ appears in \eqref{expiadicf}}\}\le1.
\]   
\end{enumerate}
It is easy to show that full $i$-th expansions exist (see Section \ref{Characteri}). Also, it is easy to see that such expansions might not be unique. 

In order to simplify the notation and make our constructions more rigid, we will require these full $i$-th expansions to satisfy certain additional conditions that will be spelled out below. For $\ell\in I$ we denote by $q_\ell$ the unique element of
$\{1,\dots,w\}$ such that $\ell\in I^{(q_\ell)}$.

For $i,\ell\in I$ we write $i\ll \ell$ if $q_i<q_\ell$. Also, we write $Q_i<_{\rm succ}Q_\ell$ if
\[
\text{either}\ \ell=i+1\mbox{ or }(\#I^{(q_i)}=\infty\text{ and }q_\ell=q_i+1).
\]
Consider elements $i,\ell\in I$ such that
\begin{equation}
Q_i<_{\rm succ}Q_\ell.\label{eq:successor}
\end{equation}
The {\bf level} of the pair $(\ell,i)$ is the integer
\begin{displaymath}
s(\ell,i):=\left\{
\begin{array}{ll}
i-\ell_{q_i}&\text{\,if }\ell=i+1\\
\max\{i-\ell_{q_i},\ell-\ell_{q_\ell}\}&\text{\,otherwise}.
\end{array}
\right.
\end{displaymath}
We say that the pair $(\ell,i)$ is {\bf neat} if the following implication holds. If $q_i\neq q_\ell$ and
$\#I^{(q_i)}=\#I^{(q_\ell)}=\infty$ then $i-\ell_{q_i}=\ell-\ell_{q_\ell}$.

Fix a polynomial $f\in K[x]\setminus\{0\}$ and consider a full $i$-th expansion \eqref{expiadicf} of $f$. Recall that by definition of full $i$-th expansions, for each $q\in\{1,\dots,w\}$ there exists at most one $i\in I^{(q)}$ such that $\tilde Q_i$ appears in \eqref{expiadicf}. If such an $i$ does exist, we will denote it by $j(f,q)$. If there is no such $i$, we will write $j(f,q)=-\infty$; we adopt the convention that $-\infty<I$.

We say that the full $i$-th expansion \eqref{expiadicf} is {\bf neat} if there exists $s\in\N$ satisfying
\[
j(f,q)-\ell_q=s\text{ for all }q\in\{1,\dots,w\}\text{ such that }\#I^{(q)}=\infty\text{ and }j(f,q)>-\infty
\]
 (recall that we have $j(f,q)-\ell_q=0$ whenever $\#I^{(q)}=1$ and $j(f,q)>-\infty$). In this situation, we will say that \eqref{expiadicf} is {\bf neat of level} $s$, provided the set of indices $q$ with $\#I^{(q)}=\infty$ and $j(f,q)>-\infty$ appearing above is not empty. If the set of such $q$ is empty, then we will say that \eqref{expiadicf} is {\bf neat of level} 0.

\begin{Obs} It is possible for \eqref{expiadicf} to be of level 0 even if the set of indices $q$ appearing in the above definition is not empty.
\end{Obs}

The first main result of this paper (Lemma \ref{lemaneat}) is that we can choose $\textbf{Q}$ as above in  such way that for every neat pair $(\ell,i)$, there exists a neat full $i$-th expansion of $\tilde Q_\ell$. More precisely, for neat pairs
$(\ell,i)$, $i,\ell\in I$, we will construct, recursively in the level $s(\ell,i)$, neat full $i$-th expansions
\begin{equation}
\tilde Q_\ell=\sum\limits_ {j=1}^rb_{\ell ij}\tilde{\textbf{Q}}^{\lambda_j}.\label{eq:fulliexpofQtildel}
\end{equation}

For each $i\in I^*$ choose an independent variable $X_i$ and let $\textbf X=\{X_i\}_{i\in I^*}$. By Proposition \ref{GenerationofOLoverOK} below, the natural map $\pi:\VR_K[{\bf X}]\rightarrow L$ defined by $X_i\mapsto \tilde Q_i$ induces an isomorphism
\[
\VR_L\simeq \VR_K[\textbf{X}]/\mathcal I,
\]
where $\mathcal I$ is the kernel of $\pi$.
\begin{Obs}
The complete set $\left(\tilde Q_i\right)_{i\in I}$ for $\nu$ provides (as described in Section \ref{completesets}) a set of generators for $\VR_L$ over $\VR_K$. The ideal $\mathcal I$ above is the ideal generated by all the relations between the $\tilde Q_i$'s. Our goal is to present a simple set of generators of $\mathcal I$.
\end{Obs}
Consider elements $i,\ell\in I$ satisfying \eqref{eq:successor}. Fix an element $b_{\ell i}\in K$ such that
\begin{equation}
v\left(b_{\ell i}\right)=-\nu_i\left(\tilde Q_\ell\right)=-\min\limits_jv\left(b_{\ell ij}\right),\label{eq:minusthevalue}
\end{equation}
where the notation is as in \eqref{eq:fulliexpofQtildel}. Set
\begin{equation}\label{equaimportm}
Q_{\ell i}=b_{\ell i}\sum\limits_jb_{\ell ij}{\textbf{X}}^{\lambda_j}\in \VR_K[\textbf X].
\end{equation}
Let
\begin{equation}
\mathcal I_1=\left(\left.b_{\ell i}X_\ell-Q_{\ell i}\ \right|\  i,\ell\in I^*, (\ell,i)\text{ is neat}\right)\VR_K[\bf X]\label{eq:defI1}
\end{equation}
and
\begin{equation}
\mathcal I_2=\left(\left.Q_{i_{\max}i}\ \right|\ i\in I^*,Q_i<_{\rm succ}Q_{i_{\max}}\right)\VR_K[\bf X].\label{eq:defI2}
\end{equation}

The main result goal of this paper to show that
\[
\mathcal I=\mathcal I_1+\mathcal I_2.
\]
This holds under one additional mild technical assumption that we describe below.  We say that $\textbf{Q}$ is \textbf{successively strongly monic} if $Q_\ell$ is $Q_i$-strongly monic for every pair $(\ell,i)$ such that $Q_i<_{\rm succ}Q_\ell$. The third main assumption we will use is that
\begin{description}
\item[Assumption 3] $\textbf{Q}$ is successively strongly monic.
\end{description}
\begin{Obs}
This assumption holds whenever $(K,v)$ is henselian and whenever $\rk\ v=1$.
\end{Obs}

The main result of this paper (Theorem \ref{generatorsCalI}) is that under \textbf{Assumptions 1--3} we have $\mathcal I=\mathcal I_1+\mathcal I_2$. The importance of this result lies in the fact that the relations in \eqref{eq:defI1} and \eqref{eq:defI2} are very well behaved; in particular, they are uniquely determined by the pair $(\ell,i)$. This will be important for the study of the module of the K\"ahler differentials for $\VR_L/\VR_K$ (work in preparation).

\textbf{Acknowledgements.} We would like to thank the anonymous referee for carefully reading and providing useful suggestions in an earlier version of this paper.

\section{Some properties of complete sets}\label{completesets}

A set $\textbf{Q}\subseteq K[x]\setminus K$ is called a {\bf complete} set for $\nu$ if for every $f\in K[x]$ there exists
$q\in\textbf{Q}$ such that
\begin{equation}\label{eqabotdegpol}
\deg(q)\leq\deg(f)\mbox{ and }\nu(f)=\nu_q(f).
\end{equation}

Let $L=K(\eta)$ and take a valuation $v$ on $L$ such that $e(L/K,v)=1$. Consider the valuation $\nu$ on $K[x]$ induced by $v$ and $\eta$ and fix a complete set $\textbf{Q}=\{Q_i\}_{i\in I}$ for $\nu$. If $g$ is the minimal polynomial of $\eta$, then we can assume that $g=Q_{i_{max}}$ for some $i_{\max}\in I$ and $\nu(Q_i)\neq \infty$ for every $i\in I^*=:I\setminus\{i_{max}\}$. For every $i\in I^*$ choose $a_i\in K$ such that $\nu(Q_i)=v(a_i)$ and set $\tilde Q_i=Q_i/a_i$.

\begin{Prop}\label{GenerationofOLoverOK}\cite[Proposition 3.5]{NovSpiAnn}
Under the above assumptions and notation we have
\[
\VR_L=\VR_K\left[\left.\tilde Q_i(\eta)\ \right|\ i\in I\right].
\]
\end{Prop}

Let $\overline K$ denote the algebraic closure of $K$. Fix an extension $\overline \nu$ of $\nu$ to $\overline K[x]$. For each $f\in K[x]\setminus K$ we define
\[
\epsilon(f):=\max\{\overline \nu(x-a)\mid a\mbox{ is a root of }f\}.
\]
A monic polynomial $Q\in K[x]\setminus K$ is called \textbf{a key polynomial for} $\nu$ if $\epsilon(f)<\epsilon(Q)$ for every $f\in K[x]\setminus K$ with $\deg(f)<\deg(Q)$.

From now on, we will fix a complete sequence of key polynomials $\textbf{Q}=\{Q_i\}_{i\in I}$ for $\nu$. This means that every element $Q_i$ is a key polynomial for $\nu$, the set $I$ is well ordered, the map $i\mapsto \epsilon(Q_i)$ is strictly increasing and the set $\{Q_i\}_{i\in I}$ is a complete set for $\nu$. In particular, $g$ is the last element of $\textbf{Q}$ and
\begin{equation}
\forall Q_i\in{\bf Q}\setminus\{g\}\mbox{ we have }\nu(Q_i)\in vL.\label{eq:nuQi<infty}
\end{equation}
The existence of such sequences of key polynomials is proved, for example, in \cite{Novspivkeypol}.

\section{Full $i$-th expansions}\label{Characteri}

We keep the notation of the previous sections.
\begin{Def} Given two valuations $\mu,\mu'$ of $K[x]$ extending $v|_K$, we write $\mu\le\mu'$ if
\begin{equation}
\mu(f)\le\mu'(f)\label{eq:m'lemu}
\end{equation}
 for all $f\in K[x]$. We say that $\mu<\mu'$ if, in addition, the inequality \eqref{eq:m'lemu} is strict for some $f\in K[x]$.
\end{Def}
The following fact is well known and easy to prove (for instance, see \cite[Corollary 3.13]{Nov11}). Given two indices $i,i'\in I^*$ with $i<i'$, we have
\begin{equation}
\nu_i<\nu_{i'}<\nu.\label{eq:nui<nui'}
\end{equation}

We will proceed to prove that for every $i\in I^*$ and $f\in K[x]$, there exists at least one full $i$-th expansion of $f$. Let
\[
f=f_{i0}+f_{i1}\tilde Q_i+\ldots+f_{in}\tilde Q_i^n
\]
be the $\tilde Q_i$-expansion of $f$. For each $f_{ij}\neq 0$ there exists $k\ll i$ such that
\begin{equation}
\nu_k(f_{ij})=\nu(f_{ij}).\label{eq:nukfij}
\end{equation}
Choose $k$ to be any element of $I$ such that $k\ll i$ and the equality \eqref{eq:nukfij} holds for every
$j\in\{0,\dots,n\}$.
Let
\[
f_{ij}=f_{jk0}+f_{jk1}\tilde Q_k+\ldots+f_{jks}\tilde Q_k^s
\]
be the $\tilde Q_k$-expansion of $f_{ij}$. We can proceed taking expansions of the ``coefficients". Since at each step the degree of the coefficients of the expansions decreases strictly, we will reach the case where these coefficients belong to $K$. Hence we obtain an expression of the form \eqref{expiadicf}. It is easy to verify that this expression satisfies the conditions (1)--(3) in the definition of full $i$-th expansion. In particular,
\begin{equation}
\nu_i(f)=\min_{1\leq j\leq r}\{v(b_j)\}.\label{eq:nui(f)=minvbj}
\end{equation}

\begin{Obs}\label{propertiesithexpansion} Full $i$-th expansions have the following properties.
\begin{enumerate}
\item Equality $\nu_k(f_{ij})=\nu(f_{ij})$ and its analogues hold at each step of our recursive process.
\item The full $i$-th expansion  is not, in general, unique. However, it is uniquely determined by the tuple
\begin{equation}
\left(k\ \left|\ \tilde Q_k\mbox{ appears in \eqref{expiadicf}}\right.\right),\label{eq:tupleappears}
\end{equation}
where we write the indices $k$ in the tuple \eqref{eq:tupleappears} in the decreasing order.
\end{enumerate}
\end{Obs}

\begin{Obs}
By construction of the full $i$-th expansion we have
\[
\deg\left(\frac{\tilde{\textbf{Q}}^{\lambda_j}}{\tilde Q_i^{\lambda_j(Q_i)}}\right)<\deg(Q_i)\mbox{ for every }j,1\leq j\leq r.
\]
\end{Obs}
\medskip

Recall that $\textbf{X}=\{X_i\}_{i\in I^*}$ is a set of indeterminates.
\medskip

\noindent{\bf Notation.}  For $i\in I^*$, denote
\[
\textbf{X}_{<i}=\{X_{i'}\}_{\substack{i'\in I^*\\i'<i}},\ \ \textbf{X}_{\le i}=
\{X_{i'}\}_{\substack{i'\in I^*\\i'\le i}}\mbox{ and }\textbf{X}_{\ge i}=\{X_{i'}\}_{\substack{i'\in I^*\\i'\ge i}}.
\]
Let 
\[
F\in K[\textbf{X}]:=K[X_i\mid i\in I^*],
\]
that is, there exist $\lambda_1,\ldots,\lambda_s\in \N^{I^*}$ and $b_1,\ldots,b_s\in K$ such that
\begin{equation}
F=\sum_{j=1}^sb_j\textbf{X}^{\lambda_j}.\label{eq:fexpansion}
\end{equation}
We write
\begin{equation}
F_{\textbf{Q}}:=F\left(\textbf{Q}\right)=\sum_{j=1}^sb_j\textbf{Q}^{\lambda_j}\mbox{ and
}F_{\tilde{\textbf{Q}}}:=F\left(\tilde{\textbf{Q}}\right)=\sum_{j=1}^sb_j\tilde{\textbf{Q}}^{\lambda_j}.\label{eq:fQexpansion}
\end{equation}

\begin{Obs}\label{eq:decompositionI*andrelations3} Take a $q\in\{1,\dots,w\}$ such that $\#I^{(q)}=\infty$.
\begin{enumerate}
\item Take an element  $i\in I^{(q)}$ and let $\ell=i+1$. Then $\tilde Q_{i+1}$ has the form
\[
b_{i+1,i}\tilde Q_{i+1}=\tilde Q_i+(A_i)_{\tilde{\textbf{Q}}}
\]
where
\begin{equation}
A_i\in\VR_K\left[{\bf X}_{<\ell_q}\right]\label{eq:Rli3}
\end{equation}
and $b_{i+1,i}\in\VR_K$.

Given an expression of the form
\begin{equation}
\sum\limits_{j=1}^rb_j{\bf\tilde Q}^{\lambda_j}\mbox{ with }b_j\in K,\label{eq:arbitraryexpansion1}
\end{equation}
consider the operation of replacing every occurrence of $\tilde Q_i$ in \eqref{eq:arbitraryexpansion1} by
\[
b_{i+1,i}\tilde Q_{i+1}-(A_i)_{\bf{\tilde Q}}.
\]
Every time $\tilde Q_i$ appears in an expansion of a polynomial in $K[x]$, it can be eliminated by the operation described above. The inequalities \eqref{eq:nui<nui'} show that the set $\tilde{\bf Q}\setminus\{Q_i\}$ is still a complete set for $\nu$.
\item  We can repeat the above procedure (replacing $\tilde{\bf Q}$ by $\tilde{\bf Q}\setminus\{Q_i\}$) infinitely many times. In other words, consider sets $\tilde I^{(q)}\subset  I^{(q)}$ such that 
\[
\#\tilde I^{(q)}=\#I^{(q)},\quad q\in\{1,\dots w\},
\]
and put $\tilde I^*=\coprod\limits_{q=1}^w\tilde I^{(q)}$ with the ordering induced from $I^*$. Then $\left\{\tilde
Q_i\right\}_{i\in\tilde I^*}$ is still a complete set for $\nu$.
\end{enumerate}
\end{Obs}

Fix an element $f\in K[x]$ and let
\begin{equation}\label{eppansnfoposjjfh}
f=f_0+f_1\tilde Q_i+\ldots+f_r\tilde Q_i^r
\end{equation}
be the $\tilde Q_i$-expansion of $f$. 
\medskip

\noindent{\bf Notation.} The subset $S_i(f)\subset\{0,\dots,r\}$ is defined by
\[
S_i(f):=\left\{j\in\{0,\dots,r\}\ \left|\ \nu\left(f_j\tilde Q_i^j\right)=\nu_i(f)\right.\right\}.
\]
\begin{Obs} The same set $S_i(f)$ can be defined analogously using the $Q_i$-expan-sion instead of the $\tilde Q_i$-expansion; the result does not depend on which of the two types of expansions we use.
\end{Obs}
Fix indices $i<\ell$ in $I$. The following definition is inspired by the work of Julie Decaup \cite{D}, but is more restrictive than hers.

\begin{Def} We say that $Q_\ell$ is an {\bf immediate successor} of $Q_i$ (denoted by $Q_i<_{\text{imm}}Q_\ell$) if
$\ell=i+1$.
\end{Def}
\begin{Obs} It may happen that $\deg\ Q_{i+1}=\deg\ Q_i$ (if $\#I^{(q_i)}=\infty$).
\end{Obs}
\begin{Def}[J. Decaup \cite{D}] We say that $Q_\ell$ is a {\bf limit successor} of $Q_i$ (denoted by $Q_i<_{\lim}Q_\ell$) if
$\#I^{(q_i)}=\infty$ and $q_\ell=q_i+1$.
\end{Def}
\begin{Obs}
By definition, $Q_i<_{\rm succ}Q_\ell$ if and only if either $Q_i<_{\text{imm}}Q_\ell$ or $Q_i<_{\lim}Q_\ell$.
\end{Obs}
Assume that $Q_i<_{\rm succ}Q_\ell$. Let $Q_\ell=\sum\limits_{j=0}^{r_{\ell i}}q_jQ_i^j$ be the $Q_i$-expansion of $Q_\ell$.
\begin{Def}\label{stronglymonic} We say that $Q_\ell$ is {\bf strongly $Q_i$-monic} if
\[
q_{r_{\ell i}}=1
\mbox{ and } r_{\ell i}\in S_i(Q_\ell).
\]
\end{Def}
\begin{Obs}\begin{enumerate}
\item If $Q_i<_{\text{imm}}Q_\ell$, then  $Q_\ell$ is strongly $Q_i$-monic.
\item The property of $Q_\ell$ being strongly $Q_i$-monic does not depend on the particular  choice of $Q_i$ and $Q_\ell$, only on their respective degrees.
\end{enumerate}
\end{Obs}
\begin{Def}\label{successivelystronglymonic} We say that $\bf Q$ is {\bf successively strongly monic} if for each pair
$(i,\ell)$ as above with $\ell\in I^*$, $Q_\ell$ is strongly $Q_i$-monic.
\end{Def}
From now on, we will assume that $\textbf{Q}$ is successively strongly monic. This condition is automatically satisfied if for instance the rank of $v$ is one or $(K,v)$ is \emph{henselian}. 
\section{Neat polynomials}\label{Immediatesuccessors}

\noindent{\bf Notation.} For $F\in K[\textbf{X}]$, let $\mu_0(F)$ denote the minimum of the values of the coefficients of
$F$:
\begin{equation}\label{eq:mu0=nui}
\mu_0(F)=\min\limits_{j\in\{1,\dots,s\}}\left\{v\left(b_j\right)\right\}
\end{equation}
in the notation of \eqref{eq:fexpansion}.

We will now deduce some properties of $b_{\ell i}$ and $Q_{\ell i}$ as in \eqref{eq:minusthevalue} and \eqref{equaimportm}. If $\ell\ne i_{\max}$, then we will choose $b_{\ell i}$ of the following special form. Since in this case $Q_\ell$ is assumed to be strongly $Q_i$-monic, after renumbering the monomials in \eqref{eq:fulliexpofQtildel} we may assume that
$\tilde{\textbf{Q}}^{\lambda_1}$ is divisible by $\tilde Q_i$ but not by any $Q_{i'}$ with $i'<i$ and that
$vb_{\ell i1}=\nu_i\left(\tilde Q_\ell\right)$. Put $b_{\ell i}:=\frac1{b_{\ell i1}}$.

Note that, in all cases, we have
\begin{equation}
v\left(b_{\ell i}\right)>0\mbox{ whenever }\ell\in I^*.\label{eq:vbli>0}
\end{equation}
\begin{Obs} By definition of $b_{\ell i}$, we have $\mu_0(Q_{\ell i})=0$ (this holds for all the pairs $(\ell,i)$ as above, including when $\ell=i_{\text{max}}$). In particular,
\[
Q_{\ell i}\in\VR_K[\bf X].
\]
\end{Obs} 

Let $(\ell,i)$ be a neat pair and fix a full $i$-th expansion \eqref{eq:fulliexpofQtildel} of $\tilde Q_\ell$. Let
$s:=s(\ell,i)$. Let S$_{\ell,i}$ denote the  statement that  the expansion \eqref{eq:fulliexpofQtildel} is neat of level $s$.

\begin{Def}\label{neat} We say that $\bf{\tilde Q}$ is {\bf neat} if S$_{\ell,i}$ holds for all the neat pairs $(\ell,i)$.
\end{Def}

\begin{Lema}\label{lemaneat} After replacing $I^*$ by a suitable subset $\tilde I^*$, it is possible to choose an expansion \eqref{eq:fulliexpofQtildel} for each neat pair $(\ell,i)$ in such a way that $\bf{\tilde Q}$ becomes neat.
\end{Lema}
\begin{proof} Let  $\text C_s$ denote the statement ``S$_{\ell,i}$ holds for all the neat pairs $(\ell,i)$ of level $s$". We will prove by induction on $s$ that C$_s$ holds for all $s$ after replacing $I^*$ by a suitable subset $\tilde I^*$ for an appopriate choice of full $i$-th expansions; this will complete the proof of the lemma.

Fix a non-negative integer $s$. Assume that
\begin{equation}
\text C_{s'}\text{ holds for all }s'<s .\label{l'i'<li}
\end{equation}
This includes the base of the induction (the case $s=0$), in which case the hypothesis \eqref{l'i'<li} is vacuously true. In this way, we will prove both the base of the induction and the induction step simultaneously.

Let $\{r_1,\dots,r_t\}$ denote the ordered set $\left\{q\in\{1,\dots,w\}\ \left|\
\# I^{(q)}=\infty\right.\right\}$. For a full $i$-th expansion \eqref{eq:fulliexpofQtildel}, a subset $T\subset\{1,\dots,t\}$ and indices $i_j\in I^{(r_j)}$, $j\in T$, we will say that {\bf \eqref{eq:fulliexpofQtildel} involves at most} $\left\{\tilde Q_{i_j}\right\}_{j\in T}$ if for all $j\in T$ we have
\[
j(Q_\ell,r_j)\in\{i_j,-\infty\}.
\]
For $j\in\{1,\dots, t\}$, let C$_{s, j}$ denote the statement ``for every neat pair $(\ell,i)$ of level $s$ satisfying $i\ge r_j$, the element $\tilde Q_\ell$ admits a full $i$-th expansion involving at most
$\tilde Q_{\ell_{r_j}+s}$, $\tilde Q_{\ell_{r_{j+1}}+s},\dots,\tilde Q_{\ell_{r_t}+s}$". We have  C$_{s,1}=$C$_s$. We will now gradually modify the set $I^*$, in the decreasing order of $j$, to ensure that C$_{s,j}$ holds for all $j\in\{1,\dots,t\}$. Once we arrive at the situation where C$_{s,1}=$C$_s$ holds, the proof of the lemma will be complete.

The statement C$_{s, t}$ holds automatically. Fix a $j\in\{1,\dots, t-1\}$ and assume that C$_{s,j+1}$ is true.

Take an integer $u\ge s$ large enough so that for every neat pair $(\ell,i)$ of level $s$ satisfying $i\ge r_j$, the element $\tilde Q_\ell$ admits a full $i$-th expansion involving at most $\tilde Q_{\ell_{r_j}+u}$, $\tilde Q_{\ell_{r_{j+1}}+s},\dots,\tilde Q_{\ell_{r_t}+s}$.
\medskip

Next, apply the procedure of Remark \ref{eq:decompositionI*andrelations3} (2) to $q=r_j$. Namely, put
\[
\tilde I^{(r_j)}=I^{(r_j)}\setminus\left\{\ell_{r_j}+ s,\dots,\ell_{r_j}+u-1\right\}.
\]
If $q\in\{1,\dots,w\}\setminus\{r_j\}$, put
\[
\tilde I^{(q)}=I^{(q)}.
\]
Finally, let $\tilde I^*=\coprod\limits_{ q=1}^w\tilde I^{( q)}$. Replacing $I^*$ by $\tilde I^*$, we arrive at the situation when C$_{s,j}$ holds. After finitely many iterations of this procedure we arrive at the situation when C$_{s,1}=$C$_s$ holds. By induction on $s$, this completes the proof of the lemma.
\end{proof}

\section{A description of the $\VR_K$-algebra $\VR_L$ in terms of generators and relations}\label{Characteri1}

The purpose of this section is to give a description of the $\VR_K$-algebra $\VR_L$ in terms of generators and relations. From now on, we will assume that $\bf{\tilde Q}$ is  neat.

Consider the natural maps
\[
\VR_K[\textbf{X}]\overset{\bf e}\lra K[x]\overset{\text{ev}_\eta}\lra L, \ \ F\mapsto F_{\tilde{\bf
Q}}\mapsto F_{\tilde{\bf Q}}(\eta).
\]
Let $\mathcal I=\text{Ker}(\text{ev}_\eta\circ{\bf e})$, so that
\[
\VR_L\simeq \VR_K[\textbf{X}]/\mathcal I.
\]

Below we will give an explicit description of $\mathcal I$ in terms of elements obtained from neat expansions of elements of $\mathbf{\tilde Q}$. This will be done using the ideals $\mathcal I_1$ and $\mathcal I_2$ defined in the Introduction. 

\begin{Obs} We have
\begin{equation}
\mathcal I_1\subset\mbox{Ker}\ {\bf e},\label{eq:I1inKere}
\end{equation}
in particular,
\begin{equation}
\mathcal I\supset\mathcal I_1.\label{eq:I1easyinclusion} 
\end{equation}
\end{Obs}

Consider elements $ i,\ell\in I^*$ such that $Q_i<_{\rm succ}Q_\ell$. We will now define an operation on
$K[\textbf{X}]\setminus\{0\}$ called the $(i,\ell)$-building. Fix a polynomial $F\in K[\textbf{X}]\setminus\{0\}$.

\begin{Obs} Since $Q_\ell$ is assumed to be strongly $Q_i$-monic, there exists a unique expression of the form
\begin{equation}
F=\sum\limits_{j=0}^da_j\left(\frac{Q_{\ell i}}{b_{\ell i}}\right)^j,\text{ where }a_j\in K[{\bf X}]\text{ and }
\deg_{X_i}a_j<\deg_{Q_i}Q_\ell.\label{eq:Qlibliexpansion}
\end{equation}
\end{Obs}

\begin{Def} The expression \eqref{eq:Qlibliexpansion} is called the $\frac{Q_{\ell i}}{b_{\ell i}}${\bf-expansion} of $F$. 
\end{Def}

\begin{Def} The $(i,\ell)${\bf-building} of $F$ is the polynomial $F_{\ell i}^{\text{bdg}}$  obtained from $F$ by substituting $X_\ell$ for $\frac{Q_{\ell i}}{b_{\ell i}}$ in the $\frac{Q_{\ell i}}{b_{\ell i}}$-expansion of $F$.
\end{Def}
\begin{Obs} We have
\begin{equation}
F\equiv F_{\ell i}^{\text{bdg}}\mod\,\mathcal I_1K[\bf X].\label{eq:bdgmodI1KX}
\end{equation}
\end{Obs}
\begin{Obs} The $(i,\ell)$-building leaves the polynomial $F$ unchanged if and only if
$\deg_{X_i}F<\frac{\deg_xQ_\ell}{\deg_xQ_i}$.
\end{Obs}
\begin{Obs}\label{eq:decompositionI*andrelations0} Take a $q\in\{1,\dots,w\}$ such that $\#I^{(q)}=\infty$. Consider a special case of the above definition where  $i\in I^{(q)}$ and $\ell=i+1$. Then $\tilde Q_{i+1,i}$ satisfies the equation
\[
b_{i+1,i}\tilde Q_{i+1,i}=X_i+A_i
\]
where
\begin{equation}
A_i\in\VR_K\left[{\bf X}_{<\ell_q}\right].\label{eq:Rli0}
\end{equation}
The $(i,i+1)$-building of $F$ consists of replacing every occurrence of $X_i$ in $F$ by $b_{i+1,i}X_{i+1}-A_i$ (recall that
$b_{i+1,i}\in\mathfrak m_L$).
\end{Obs}
We will now define the class of neat polynomials in $K[\bf X]$.  

\begin{Def} For a polynomial
\[
F\in K[\bf X]
\]
as in \eqref{eq:fexpansion} and $i\in I^*$, we  will say that $X_i$ {\bf appears in} $F$ if $X_i\ \left|\ \bf X^{\lambda_j}\right.$ (that is, if $\lambda_j(i)>0$) for some $j\in\{1,\dots,s\}$.
\end{Def}

Fix a polynomial $F\in K[{\bf X}]\setminus\{0\}$.
\begin{Def}\label{neatdefinition} We say that $F$ is {\bf neat} if the following three conditions hold.
\begin{enumerate}
\item For each $q\in\{1,\dots,w\}$ there exists at most one $i\in I^{(q)}$ such that $X_i$ appears in $F$. If such an $i$ does exist, we will denote it by $j(F,q)$. If there is no such $i$, we will write $j(F,q)=-\infty$; we adopt the convention that
$-\infty<I$.
\item There exists $s\in\N$ satisfying $j(F,q)-\ell_q=s$ for all $q\in\{1,\dots,w\}$ such that $\#I^{(q)}=\infty$ and $j(F,q)>-\infty$ (recall that we have $j(F,q)-\ell_q=0$ whenever $\#I^{(q)}=1$ and $j(F,q)>-\infty$).
\item Let $i$ be the unique element of $I^*$ such that
\[
F\in K[{\bf X}_{\le i}]\setminus K[{\bf X}_{<i}].
\]
Then for each $i'\in I_{<i}$ we have $\deg_{X_{i'}}F<\frac{(\deg_xQ_{i'})_+}{\deg_xQ_{i'}}$.
\end{enumerate}
In this situation, we will say that $F$ is {\bf neat of level} $s$, provided that the set of indices $q$ with
$\#I^{(q)}=\infty$ and $j(F,q)>-\infty$ appearing in (2) of this definition is non-empty. If the set of such $q$ is empty, we will say that $F$ is {\bf neat of level} 0.
\end{Def}

\begin{Obs} Recall that we are assuming that $\bf{\tilde Q}$ is neat. By definition, this implies that all the polynomials $Q_{\ell i}$ are neat and the polynomial $b_{\ell i}X_\ell-Q_{\ell i}$ is neat whenever $i\ll\ell$.
\end{Obs}

\begin{Prop}\label{neatuniqueness} Fix an index $i\in I^*$ and two neat polynomials
\[
F,\tilde F\in K[{\bf X}_{\le i}]\setminus K[{\bf X}_{<i}]
\]
of the same level. If $F\equiv\tilde F\mod\mathcal I_1K[\bf X]$ then $F=\tilde F$.
\end{Prop}
\begin{proof} We prove the proposition by contradiction. Assume that
\[
F-\tilde F\ne0.
\]
Write
\begin{equation}
F-\tilde F=\sum\limits_{j=0}^rF_jX_{i'}^j,\label{eq:fXiexpansion}
\end{equation}
where $i'\in I_{\le i}$, $F_j\in K[{\bf X}_{<\ell_{q_{i'}}}]$ (this is where we use the hypothesis that $F$ and
$\tilde F$ are neat of the same level), and $F_r\ne0$. Conditions (1)--(3) of Definition \ref{neatdefinition} imply that
\begin{equation}
\deg_x(F_j)_{\bf{\tilde Q}}<\deg_xQ_{i'}\quad\text{for all }j\in\{0,\dots,r\}.\label{eq:upperbounddegQi}
\end{equation}
It follows from \eqref{eq:upperbounddegQi} that
\[
\deg_x\left((F_r)_{\bf{\tilde Q}}Q_{i'}^r\right)>\deg_x\left((F_j)_{\bf{\tilde Q}}Q_{i'}^j\right)
\]
for all $j\in\{0,\dots,r-1\}$. This contradicts \eqref{eq:I1inKere} in view of the fact that
\[
F-\tilde F\in\mathcal I_1K[{\bf X}].
\]
The proposition is proved.
\end{proof}
Consider elements $i,\ell\in I$ such that
\begin{equation}
Q_i<_{\rm succ}Q_\ell.\label{eq:Qi'succQell}
\end{equation}
Recall that by definition of full $i$-th expansions, for every $q'\in\{1,\dots,q_\ell\}$ there is at most one $i'\in I^{(q')}$ such that $X_{i'}$ appears in $Q_{\ell i}$. As above, we denote this unique $i'$ by $j(Q_{\ell i},q')$. Let us write $j(\ell,i,q'):=j(Q_{\ell i},q')$. Note that if $\# I^{(q')}=1$ and $j(\ell,i,q')>-\infty$, then $j(\ell,i,q')=\ell_{q'}$, the unique element of $I^{(q')}$.
\medskip

Next, we define the operation of the total $s$-building of a polynomial $F$ for $s\in\N$.
\smallskip

\noindent{\bf Notation.} For $q\in\{1,\dots,w\}$ and $i\in I^{(q)}$, let ${\bf X}_{q,\le i}:=(X_{i'})_{\ell_q\le i'\le i}$.
\medskip

Take a polynomial $F\in K[{\bf X}]\setminus\{0\}$ and let $q$ be the element of $\{1,\dots,w\}$ such that $F\in K\left[{\bf
X}_{<\ell_{q+1}}\right] \setminus K\left[{\bf X}_{<\ell_q}\right]$. Fix a natural number $s$.

\begin{Obs} By Proposition \ref{neatuniqueness}, there is at most one neat polynomial $F_s$ of level $s$ in $K\left[{\bf
X}_{<\ell_{q+1}}\right]\setminus K\left[{\bf X}_{<\ell_q}\right]$ such that
\begin{equation}
F-F_{s}\in\mathcal I_1K[\bf X].\label{eq:totibuidingmodI1}
\end{equation}
\end{Obs}
\begin{Def}\label{Deftotali-building} {\bf The total $s$-building} $F_s$ of $F$ is the unique neat polynomial of level $s$ satisfying \eqref{eq:totibuidingmodI1} (if it exists).
\end{Def}

Keep the above notation. The purpose of the next proposition is to show that $F_{s}$ exists provided the integer $s$ is sufficiently large, namely, provided that for every $X_i$ appearing in $F$ we have
\begin{equation}\label{eq:iinIqsufflarge}
i-\ell_{q_i}\le s.
\end{equation}
Under this assumption the total $s$-building $F_{s}$ can be obtained from $F$ by finitely many applications of
$(i,\ell)$-buildings with different neat pairs $(i,\ell)$. We will describe an explicit algorithm for doing this. In order to show that our algorithm terminates, we need to impose a partial ordering on the set $K[\bf X]$, an ordering which we now define.

\begin{Def} For $i\in I^*$, the {\bf virtual degree of} $X_i$, denoted by $\vdeg\ X_i$, is the quantity $\deg_xQ_i$. For a monomial $\prod\limits_{j=1}^sX_{i_j}^{\gamma_j}$, $\gamma_j\in\N$, we define its virtual degree to be
$\vdeg\prod\limits_{j=1}^sX_{i_j}^{\gamma_j}:=\sum\limits_{j=1}^s\gamma_j\,\vdeg\ X_{i_j}$. The virtual degree of a polynomial in $\bf X$ is defined to be the maximal virtual degree of its monomials. For a non-negative integer $k$, a polynomial is said to be {\bf virtually homogeneous of degree} $k$ if all the monomials appearing in it have virtual degree $k$.
\end{Def}
We now define the partial ordering $\prec$ on $K[{\bf X}]$. Order the $X_i$ is the increasing order of $ i$ (the variable $X_{0}$ comes first). Let $\prec_{\text{lex}}$ denote the lexicographical ordering on the set of monomials in $\bf X$. Furthermore, given two polynomials $F$ and $G$, we say that $F\prec_{\text{lex}}G$ if the ordered list of all the monomials of $F$, written in the decreasing order with respect to $\prec_{\text{lex}}$, is lexicographically smaller than the list of the monomials of $G$ (if two lists of monomials do not have equal length, we complete the shorter list by adding a suitable number of zeroes at the end). Finally, define the partial ordering $\prec$ on $K[\bf X]$ as follows. Given two distinct polynomials $F,G\in K[\bf X]$, write $F=\sum\limits_{j=0}^sF_j$ and $G=\sum\limits_{j=0}^sG_j$, where:
\begin{enumerate}
\item for each $j$, $F_j$ and $G_j$ are virtually homogeneous of degree $j$
\item some of the $F_j$ are allowed to be 0, including the leading terms $F_s$ and $G_s$ (without loss of generality, we may assume that at least one of $F_s$ and $G_s$ is different from 0).
\end{enumerate}
We say that $F\prec G$ if there exists $j\in\{1,\dots s\}$ such that $F_k=G_k$ for all $k\in\{j+1,\dots,s\}$ and
$F_j\prec_{\text{lex}}G_j$.
\begin{Obs}
Consider elements $i,\ell\in I^*$ such that $Q_i<_{\rm succ}Q_\ell$. We have
$\vdeg\left(b_{\ell i}X_\ell-Q_{\ell i}\right)=\vdeg\ X_\ell=\deg_xQ_\ell$. Thanks to the strongly monic property, the polynomial $b_{\ell i}X_\ell-Q_{\ell i}$ contains two monomials of virtual degree $\deg_xQ_\ell$, namely, $b_{\ell i}X_\ell$ and $X_i^{r_{\ell i}}$, where the notation is as in Definition \ref{stronglymonic}. All the remaining monomials have strictly smaller virtual degrees. This implies that applying the operation of $(i,\ell)$-building to a polynomial $F$ strictly decreases it with respect to $\prec$, provided $\deg_{X_i}F\ge\frac{\deg_xQ_\ell}{\deg_xQ_i}$
\end{Obs}

\begin{Prop} Take a polynomial $F\in K[{\bf X}]\setminus\{0\}$ and a natural number $s$ and assume that \eqref{eq:iinIqsufflarge} holds. The total $s$-building $F_{s}$ exists and can be obtained from $F$ after finitely many applications of $(i,\ell)$-buildings with different neat pairs $(i,\ell)$, all of level at most $s$.
\end{Prop}
\begin{proof} If the polynomial $F$ is its own total $s$-building, there is nothing to prove. Assume this is not the case.

Let $q$ be the element of $\{1,\dots,w\}$ such that $F\in K\left[{\bf X}_{<\ell_{q+1}}\right] \setminus K\left[{\bf
X}_{<\ell_q}\right]$. Let
\[
\Theta:=\bigcup\limits_{q'=1}^q\left\{\left.i\in I^{(q')}\ \right|\ i- \ell_{q'}\le s\right\}\mbox{ and }{\bf X}_\Theta=\{X_i\ |\ i\in\Theta\}.
\]
All of our building operations will take place inside the ring $K[{\bf X}_\Theta]$. To say that $F$ does not equal its total
$s$-building is equivalent to saying that there exists a neat pair $(\ell,i)$ with $\ell,i\in\Theta$, such that
$\deg_{X_{i}}F\ge\frac{\deg_xQ_\ell}{\deg_xQ_{i}}$. In this case, perform an $(i,\ell)$-building. This operation strictly decreases  $F$ with respect to $\prec$. If the $(i,\ell)$-building of $F$ is not the total
$s$-building, iterate the procedure. Since the polynomial $F$ cannot decrease indefinitely, the process must stop at the total $ s$-building $F_{s}$ of $F$, as desired.
\end{proof}
\begin{Obs}\label{5remarks}
\begin{enumerate}
\item  Write 
\[
F_{s}=\sum\limits_{\gamma\in\N}F_{\gamma}({\bf X}_{<\ell_q+s}){\bf X}_{\ell_q+s}^\gamma,
\]
where
\[
F_{\gamma}({\bf X}_{<\ell_q+s})\in K[{\bf X}_{<\ell_q+s}].
\]
Then (1) and (3) of Definition \ref{neatdefinition} imply that
\[
\nu_{\ell_q+s}\left((F_\gamma)_{\tilde{\bf Q}}\right)=\nu\left((F_\gamma)_{\tilde{\bf Q}}\right).
\]
\item By \eqref{eq:vbli>0}, the $(i,\ell)$-building operation does not decrease the quantity $\mu_0(F)$. In particular, if $F\in\VR_K[\textbf{X}]$, then $F_{s}\in\VR_K[\textbf{X}]$ and the congruence \eqref{eq:bdgmodI1KX} becomes
\begin{equation}
F\equiv F_{ s}\mod\,\mathcal I_1.\label{eq:bdgmodI1}
\end{equation}
\item By part (1) of this remark and equation \eqref{eq:nui(f)=minvbj}, we have
\[
\mu_0(F_{s})=\nu_{\ell_q+s}\left(F_{\bf{\tilde Q}}\right).
\]
\end{enumerate}
\end{Obs}

\begin{Teo}\label{generatorsCalI} We have
\begin{equation}
\mathcal I=\mathcal I_1+\mathcal I_2.\label{eq:I1+I2} 
\end{equation}
\end{Teo}
\begin{proof} Consider the commutative diagram
\begin{equation}
\xymatrix{\VR_K[\textbf{X}]&&\longrightarrow&&\VR_L\\
\downarrow&&&&\downarrow\\
K[\textbf{X}]&\overset{\bf e}\longrightarrow&K[x]&\overset{\text{ev}_\eta}\lra&L}\label{eq:commutative}
\end{equation}
Let $\bar{\mathcal I}:=\ker\,(\text{ev}_\eta\circ{\bf e})$. Since the vertical arrows in \eqref{eq:commutative} are injections, we have
\begin{equation}
\mathcal I=\bar{\mathcal I}\cap\VR_K[\textbf{X}].\label{eq:tildeIcontractstoI}
\end{equation}
We claim that
\begin{equation}
\bar{\mathcal I}=(\mathcal I_1+(g(X_0)))K[\textbf{X}].\label{eq:generatorsItilde}
\end{equation}
Indeed, we have
\begin{equation}
\mathcal I_1K[\textbf{X}]=\ker\,(\bf e)\label{eq:Kerbolde}
\end{equation}
and
\begin{equation}
(g(x))K[x]=\ker\,(\text{ev}_\eta).\label{eq:Keeta}
\end{equation}
Moreover, every element $F\in K[\bf X]$ is congruent modulo $\mathcal I_1K[\textbf{X}]$ to a unique element $\bar F\in K[X_0]$ and restricting $\bf e$ to $K[X_0]$ induces an  isomorphism
\begin{equation}
{\bf e}\left|_{K[X_0]}\right.:K[X_0]\cong K[x]\label{eq:isomorphismKX0Kx}
\end{equation}
that maps $X_0$ to $x$ and $g(X_0)$ to $g(x)$. Formulae \eqref{eq:Kerbolde}--\eqref{eq:isomorphismKX0Kx} show that
\begin{equation}
\bar{\mathcal I}\supset(\mathcal I_1+(g(X_0)))K[\textbf{X}],\label{eq:righthandsidecontainedinleft}
\end{equation}
in particular,
\begin{equation}
\mathcal I_1K[\textbf{X}]\subset\bar{\mathcal I}.\label{eq:I1inIbar}
\end{equation}
To prove the opposite inclusion in \eqref{eq:righthandsidecontainedinleft}, consider an element
\begin{equation}
F\in\bar{\mathcal I}.
\end{equation}
By \eqref{eq:Keeta}, \eqref{eq:isomorphismKX0Kx} and \eqref{eq:I1inIbar}, we have $\bar F\in(g(X_0))K[X_0]$. This proves the equality \eqref{eq:generatorsItilde}.
\medskip

\noindent{\bf Notation.} For $i\in I^*$ with $Q_i<_{\rm succ}g$, let $h_i:=b_{i_{\max}i}$ and $g^{(i)}=h_ig(X_0)$.
\medskip

For $F\in K[\textbf{X}]$ and $ i,\ell\in I^*$ such that $Q_i<_{\rm succ}Q_\ell$, we will now define an operation inverse to
$(i,\ell)$-building: the $(i,\ell)$-reduction.

\begin{Def} The $(i,\ell)${\bf-reduction} of $F$ is the polynomial $F_{\ell i}^{\text{red}}$ obtained from $F$ by substituting
$\frac{Q_{\ell i}}{b_{\ell i}}$ for $X_\ell$.
\end{Def}

After finitely many applications of $(i,\ell)$-reductions with different neat pairs $(i,\ell)$, every polynomial $F$ can be turned into a polynomial $F^{\text{tred}}(X_0)\in K[X_0]$ (called {\bf the total reduction} of $F$). We have
\begin{equation}
F\equiv F^{\text{tred}}(X_0)\mod\,\mathcal I_1K[\textbf{X}].\label{eq:fcongtotredmodI}
\end{equation}
Since every element of $K[X_0]$ is neat, by Proposition \ref{neatuniqueness} the polynomial $F^{\text{tred}}(X_0)$ depends only on $F$ and not on the specific chain of $(i,\ell)$-reductions used to construct it.

\begin{Obs}\label{substituteX0} In fact, $F^{\text{tred}}(X_0)$ can be obtained from $F$ by substituting $\tilde Q_i(X_0)$ for $X_i$, for each variable $X_i$ appearing in $F$. In other words, we have the equality $F^{\text{tred}}(X_0)=F_{\bf\tilde
Q}(X_0)$. This is seen by induction on the maximal index $i$ such that $X_i$ appears in $F$.
\end{Obs}

\begin{Obs}\label{partialorderingprec}\begin{enumerate}
\item Fix indices $i,i'\in I^{(w)}$, $i<i'$. We have $Q_i<_{\rm succ}g$. Let $s=i'-\ell_w$. By Proposition \ref{neatuniqueness}, $\frac{Q_{i_{\max} i'}}{b_{i_{\max} i'}}$ is the total $s$-building of
$\frac{Q_{i_{\max} i}}{b_{i_{\max} i}}$.
\item Keep the notation of part (1) of this remark. In view of part (2) of Remark \ref{5remarks}, we have
\[
-vb_{i_{\max}i'}=\mu_0\left(\frac{Q_{i_{\max}i'}}{b_{i_{\max}i'}}\right)>\mu_0\left(\frac{Q_{i_{\max} i}}{b_{i_{\max}i}}\right)=-vb_{i_{\max}i}.
\]
In other words, $vb_{i_{\max}i'}<vb_{i_{\max}i}$.
\end{enumerate}
\end{Obs}
\medskip

To prove the  inclusion $\supset$ in \eqref{eq:I1+I2}, first note that, obviously,
\begin{equation}
\mathcal I_1\subset\mathcal I.\label{eq:I1inI}
\end{equation}
To prove that $\mathcal I_2\subset\mathcal I$, fix an $i\in I^*$ with $Q_i<_{\rm succ}g$ and consider the element
\[
Q_{i_{\text{max}}i}\in\mathcal I_2.
\]
Let $s$ denote the level of $Q_{i_{\text{max}}i}$. We have
\begin{equation}
g^{(i)}\in K[X_0]\cap\bar{\mathcal I}.\label{eq:inI}
\end{equation}
By Remark \ref{substituteX0}, $g^{(i)}$ is nothing but the total reduction of $Q_{i_{\text{max}}i}$; in particular, these two polynomials are congruent modulo $\mathcal I_1K[{\bf X}]$.  Since $Q_{i_{\max}i}$ is the unique neat polynomial of level $s$ involving the variable $X_{i}$ and congruent to $g^{(i)}$ modulo $\mathcal I_1K[{\bf X}]$ (by Proposition \ref{neatuniqueness}), we have
\begin{equation}
Q_{i_{\max}i}=g_{ s}^{(i)}.\label{eq:totalibdgofg}
\end{equation}
By \eqref{eq:minusthevalue}, we have $\mu_0\left(Q_{i_{\text{max}}i}\right)=0$, so $Q_{i_{\text{max}}i}\in\VR_K[\textbf{X}]$. Combining this with \eqref{eq:tildeIcontractstoI}, \eqref{eq:generatorsItilde}, \eqref{eq:inI} and \eqref{eq:totalibdgofg}, we obtain
\[
Q_{i_{\text{max}}i}\in\mathcal I,
\]
as desired. This completes the proof of the inclusion $\supset$ in \eqref{eq:I1+I2}.

Let us prove the inclusion $\subset$. Take an element $F\in\mathcal I$. Let 
\[
s:=\max\{i-\ell_{q_i}\ |\ i\in I^*,X_i\text{ appears in }F\}.
\]
Replacing $F$ by its total $s$-building for a sufficiently large $s$, we may assume, without loss of generality, that $F$ is neat of level $s$. We have
\begin{equation}
\mu_0\left(F\right)\ge0.\label{eq:nuifpositive}
\end{equation}
Let
\begin{eqnarray*}
i&=&\ell_w\ \ \qquad\text{if }\#I^{(w)}=1\quad\text{and}\\
&=&\ell_w+s\quad\text{if }\#I^{(w)}=\infty.
\end{eqnarray*}
By \eqref{eq:tildeIcontractstoI} and \eqref{eq:generatorsItilde}, $F^{\text{tred}}\in(g(X_0))K[X_0]=\left(g^{(i)}\right)K[X_0]$. Write
\begin{equation}
F^{\text{tred}}=g^{(i)}R,\label{eq:ftre=gr}
\end{equation}
where
\[
R\in K[X_0].
\]
By \eqref{eq:fcongtotredmodI} and \eqref{eq:bdgmodI1KX}, we have
\begin{equation}
F\equiv g_{s}^{(i)}R_{s}\mod\mathcal I_1K[X].\label{eq:fcongginrimodI1}
\end{equation}
We have
\begin{equation}
\nu_i\left(g^{(i)}(x)\right)=0\label{eq:nuigi>0}
\end{equation}
by definitions. By \eqref{eq:nuifpositive} and Remark \ref{5remarks} (3), we have
\begin{equation}
\nu_i\left(F^{\text{tred}}(x)\right)\ge0.\label{eq:mu0(g)positive}
\end{equation}
From \eqref{eq:ftre=gr}, \eqref{eq:nuigi>0} and \eqref{eq:mu0(g)positive}, we obtain $\nu_i(R(x))\ge0$. Applying Remark  \ref{5remarks} (3) once again, this time to the polynomial $R$, we get
\begin{equation}
\mu_0(R_{s})\ge0.\label{eq:nuir>0}
\end{equation}
Thus $R_{s}\in\VR_K[{\bf X}_{\le i}]\subset\VR_K[\bf X]$.

By Remark \ref{substituteX0}, $g^{(i)}$ is the total reduction of both the polynomials $Q_{i_{\max}i}$ and
$\left(Q_{i_{\max}i}\right)_s$, in particular all three polynomials are congruent to each other and to $g_{ s}^{(i)}$ modulo
$\mathcal I_1$. Since $g_{ s}^{(i)}$ and $(Q_{i_{\max}i})_s$ are neat polynomials of level $s$ congruent to each other mod
$\mathcal I_1K[{\bf X}]$ and involving the variable $X_i$, they are equal by Proposition \ref{neatuniqueness}. Since
$g_{ s}^{(i)}=(Q_{i_{\max}i})_s\in\mathcal I_1+\mathcal I_2$ and in view of \eqref{eq:fcongginrimodI1}, this proves that $F\in\mathcal I_1+\mathcal I_2$, which is what we wanted to show. This completes the proof of Theorem \ref{generatorsCalI}.
\end{proof}

We end the paper with a proposition and some remarks describing certain relations between the generators of
of $\mathcal I_2$, as well as relations between $X_i$ and $X_{i'}$, where $q_i=q_{i'}$. These facts will be used in a forthcoming paper in which we calculate the $\VR_L$-module $\Omega_{\VR_L/\VR_K}$.
\medskip

\noindent{\bf Notation.} For $\ell\in I^*$, define the ideals $\mathcal I_{1,<\ell}\subset\mathcal
I_{1\ell}\subset\mathcal I_1$ and $\mathcal I_{2\ell}\subset\mathcal I_2$ by
\begin{equation}
\mathcal I_{1\ell}=\left(\left.b_{\ell'i}X_{\ell'}-Q_{\ell'i}\ \right|\ i,\ell'\in I^*,\ell'\le\ell,(\ell',i)\text{ is neat}\right),\label{eq:defI1l}
\end{equation}
\begin{equation}
\mathcal I_{1,<\ell}=\left(\left.b_{\ell'i}X_{\ell'}-Q_{\ell'i}\ \right|\ i,\ell'\in I^*,\ell'<\ell,(\ell',i)\text{ is neat}\right),\label{eq:defI1<l}
\end{equation}\begin{equation}
\mathcal I_{2\ell}=\left(\left.Q_{i_{\max}i}\ \right|\ i\in I^*,i\le\ell,Q_i<_{\rm succ}Q_{i_{\max}}\right)\label{eq:defI2l}
\end{equation}
and
\begin{equation}
\mathcal I_{2,<\ell}=\left(\left.Q_{i_{\max}i}\ \right|\ i\in I^*,i<\ell,Q_i<_{\rm succ}Q_{i_{\max}}\right).\label{eq:defI2<l}
\end{equation}
Put
\[
\mathcal I_\ell=\mathcal I_{1\ell}+\mathcal I_{2\ell}\mbox{ and }\mathcal I_{<\ell}=\mathcal I_{1,<\ell}+\mathcal I_{2,<\ell}.
\]

\begin{Obs}\label{5remarks2} With this notation, for an element $F\in K[{\bf X}_{\le i}]$, formula \eqref{eq:bdgmodI1KX} and formula \eqref{eq:bdgmodI1} of Remark \ref{5remarks} (2) can be rewritten as
\begin{equation}
F\equiv F_{s}\mod\,\mathcal I_{1i}K[{\bf X}_{\le i}].\label{eq:bdgmodI1KXi}
\end{equation}
and
\begin{equation}
F\equiv F_{s}\mod\,\mathcal I_{1i},\label{eq:bdgmodI1i}
\end{equation}
respectively.
\end{Obs}
\begin{Obs} By definition, we have $\mathcal I_{2\ell}=\mathcal I_{2,<\ell}=(0)$ whenever $\ell<\ell_w$.
\end{Obs}
The set of generators
\[
\left\{\left.Q_{i_{\text{max}}i}\ \right|\  i\in I^*,Q_i<_{\rm succ}g\right\}
\]
of $\mathcal I_2$ given in \eqref{eq:defI2} is redundant in the sense that
$Q_{i_{\text{max}}i}\in\left(Q_{i_{\text{max}}i'}\right)+\mathcal I_1$ whenever $i'\ge i$. Proving this is the goal of the next proposition.

\begin{Prop}\label{redundantgenerators} Consider indices $i<i'$ in $I^{(w)}$ (so that $Q_i<_{\lim}Q_{i_{\max}}$). We have
\[
Q_{i_{\max}i}\in\left(Q_{i_{\max}i'}\right)+\mathcal I_{1i'}.
\]
\end{Prop}
\begin{proof} Put  $s:=i'-\ell_w$. Recall that $\left(\frac{Q_{i_{\max}i}}{b_{i_{\max}i}}\right)_{s}$ denotes the total $s$-building of $\frac{Q_{i_{\max}i}}{b_{i_{\max}i}}$. By Remark \ref{5remarks2}, we have
\begin{equation}\label{eq:Qli'QliI1i'K}
\frac{Q_{i_{\max}i'}}{b_{i_{\max}i'}}=\left(\frac{Q_{i_{\max}i}}{b_{i_{\max}i}}\right)_{ s}\equiv
\frac{Q_{i_{\max}i}}{b_{i_{\max}i}}\mod\,\mathcal I_{1i'}K[{\bf X}_{\le i'}].
\end{equation}
By Remark \ref{partialorderingprec} (2), $v\left(b_{i_{\max}i'}\right)<v\left(b_{i_{\max}i}\right)$, so
$\frac{b_{i_{\max}i}}{b_{i_{\max}i'}}Q_{i_{\max}i'}\in\VR_K[{\bf X}_{\le i'}]$ and, by \eqref{eq:bdgmodI1i} and \eqref{eq:Qli'QliI1i'K}, we have
\begin{equation}\label{eq:Qli'QliI1i'}
\frac{b_{i_{\max}i}}{b_{i_{\max}i'}}Q_{i_{\max}i'}\equiv Q_{i_{\max}i}\mod\,\mathcal I_{1i'}.
\end{equation}
The proposition follows immediately from this.
\end{proof}
\begin{Obs}\label{eq:decompositionI*andrelations} Given $q\in\{1,\dots,w\}$ and $i,i'\in I^{(q)}$ with $i<i'$, we have
\[
X_i\in\VR_KX_{i'}+\mathcal I_{1,<\ell_q}.
\]
\end{Obs}


\begin{thebibliography}{99}
\bibitem{C}  S. D. Cutkosky, \emph{K\"ahler differentials, almost mathematics and deeply ramified fields}, arXiv:2310.09581 (2024).

\bibitem{CK} S. D. Cutkosky and F.-V. Kuhlmann, \emph{K\"ahler differentials of extensions of valuation rings and deeply ramified fields},  arXiv:2306.04967v1 (2023).

\bibitem{CKR} S. D. Cutkosky, F.-V. Kuhlmann and  A. Rzepka, \textit{On the computation of Kähler differentials and characterizations of Galois extensions with independent defect}, Math. Nachr. \textbf{298}, Issue 5 (2025), 1549--1577.


\bibitem{D} J. Decaup, \textit{Simultaneous Monomialization}, Contemporary Mathematics, \textbf{1 (5)} (2020), 272--346.







\bibitem{Nov11} J. Novacoski, \textit{On MacLane-Vaqui\'e key polynomials}, J. Pure Appl. Algebra \textbf{225} (2021), 106644.

\bibitem{NovSpiAnn} J. Novacoski and M. Spivakovsky, \textit{K\"ahler differentials, pure extensions and minimal key polynomials}, Arxiv:2311.14322 (2024).

\bibitem{Novspivkeypol} J. Novacoski and M. Spivakovsky, \textit{Key polynomials and pseudo-convergent sequences}, J. Algebra \textbf{495} (2018), 199--219.

\bibitem{Tei} B. Teissier, \textit{Overweight deformations of affine toric varieties and local uniformization}, Valuation Theory in Interaction, EMS Series of Congress Reports (2014), 474--565.

\bibitem{Tei2} B. Teissier, \textit{Valuations, deformations, and toric geometry}. Valuation theory and its applications, Vol. II (Saskatoon, SK, 1999), Fields Inst. Commun. \textbf{33} (2003), 361--459.
\end{thebibliography}
\end{document}